\documentclass[12pt]{amsart}
 \usepackage{a4}
\usepackage[T2A,T1]{fontenc}
\usepackage[utf8]{inputenc}
\usepackage[english]{babel}
\usepackage{amsmath}
\usepackage{amssymb}
\usepackage{amsthm}
\usepackage{mathrsfs}
\usepackage{tikz}
\usetikzlibrary{shapes,calc,patterns}
\usepackage{subfig}
\usepackage{newlfont}
\usepackage{graphicx}
\usepackage{fancyhdr}
\usepackage{indentfirst}
\usepackage{placeins}
\usepackage{eurosym}
\usepackage{subfig}
\usepackage{url}
\usepackage{newlfont}
\usepackage{xcolor}
\usepackage{listings}
\usepackage{quoting}
\usepackage{enumerate}
\usepackage[all]{xy}
\usepackage{braket}
\usepackage{mleftright}
\usepackage{array}
\usepackage{tabularx}
\usepackage{mathtools}
\usepackage{bm}
\usepackage{booktabs}

\theoremstyle{definition}
\usepackage{bm}
\usepackage{booktabs}
\newtheorem{definition}{Definition}[section]
\theoremstyle{plain}
\newtheorem{theorem}[definition]{Theorem}
\theoremstyle{plain}
\newtheorem{proposition}[definition]{Proposition}
\theoremstyle{plain}
\newtheorem{lemma}[definition]{Lemma}
\theoremstyle{plain}
\newtheorem{corollary}[definition]{Corollary}
\theoremstyle{definition}

\theoremstyle{remark}

\theoremstyle{remark}

\theoremstyle{definition}

\theoremstyle{definition}

\numberwithin{equation}{section}

\DeclareMathOperator{\Alg}{Alg}%
\DeclareMathOperator{\Clo}{Clo}%
\DeclareMathOperator{\POL}{Pol}%

\newcommand{\setofclone}[1]{\mathcal{L}_{#1}}
\renewcommand{\star}{*}

\newcommand{\finset}[1]{[{#1}]}
\DeclareMathAlphabet{\mathbfsl}{OT1}{cmr}{bx}{it}
\newcommand{\tupBold}[1]{\mathbfsl{#1}}
\renewcommand{\vec}[1]{\tupBold{#1}}

\newcommand{\ab}[1]{{\mathbf{#1}}}

\newcommand{\ari}[1]{_{#1}}
\newcommand{\arii}[1]{^{[#1]}}

\newcommand{\funset}[1]{\mathcal{#1}}
\newcommand{\restrict}[1]{|_{#1}}

\newcommand{\numberset}{\mathbb} %
\newcommand{\N}{\numberset{N}}

\newcommand{\Z}{\numberset{Z}}
\DeclarePairedDelimiter{\card}{\lvert}{\rvert}%
\newcommand{\potenza}[1]{\mathcal{P}(#1)} %
\newcolumntype{C}[1]{>{\centering}p{#1}} %

\DeclareMathAlphabet{\mathbfsl}{OT1}{cmr}{bx}{it}
\DeclareMathAlphabet{\mathsc}{OT1}{cmr}{m}{sc}

\DeclareMathOperator{\algeq}{\sim_{\text{alg}}}
\newcommand{\algclo}[2]{\text{V}_{#1}(#2)}
\newcommand{\projections}[1][A]{\mathcal{J}_{#1}}%
\newcommand{\essenar}[1]{\text{essArity}(#1)}
\title{On the number of universal algebraic geometries}
\author{Erhard Aichinger}
\address[Erhard Aichinger]{
Institut für Algebra,
Johannes Kepler Universität Linz, Altenberger Straße 69, 4040 Linz, Austria}
\email{erhard@algebra.uni-linz.ac.at}

\author{Bernardo Rossi}
\address[Bernardo Rossi]{
Institut für Algebra,
Johannes Kepler Universität Linz, Altenberger Straße 69, 4040 Linz, Austria}
\email{bernardo.rossi@jku.at}

\thanks{Supported by the Austrian Science Fund (FWF):~P33878.}
\date{\today}
\begin{document}
\begin{abstract}
The \emph{algebraic geometry} of a universal algebra $\mathbf{A}$ is defined as the collection of solution sets of systems of term equations. Two algebras $\mathbf{A}_1$ and $\mathbf{A}_2$ are called \emph{algebraically equivalent} if they have the same algebraic geometry. We prove that on a finite set $A$ with $\lvert A \rvert >3$ there are countably many algebraically inequivalent Mal'cev algebras and that on a finite set $A$ with $\lvert A \rvert >2$ there are continuously many algebraically inequivalent algebras.  
\end{abstract}
\maketitle

\section{Introduction} 
Universal algebraic geometry, introduced in \cite{Plo98,DanMyaRem12}, is based on the notion of \emph{algebraic sets}. Given a universal algebra $\ab{A}=(A;(f_i)_{i\in I})$,  a subset $B$ of $A^n$ is \emph{algebraic} if it is the solution set of a system of (possibly infinitely many) term equations in the language of $\ab{A}$. Following \cite{Pin17a}, we denote the collection of the algebraic subsets of $A^n$ by $\Alg\ari{n}\ab{A}$ and we define the \emph{universal algebraic geometry} of $\ab{A}$ by $\Alg\ab{A}:=\bigcup_{n\in\N}\Alg\ari{n}\ab{A}$. Clearly, for a universal algebra $\ab{A}$, $\Alg\ab{A}$ is completely determined by its clone of term functions (cf. \cite{BurSan81,PosKal79}). We define the universal algebraic geometry of a clone $\funset{C}$ on a set $A$ by $\Alg\funset{C}:=\Alg(A;\funset{C})$. Following \cite{Pin17a}, we say that two clones $\funset{C}$ and $\funset{D}$ on the same set $A$ are algebraically equivalent, $\funset{C}\algeq\funset{D}$, if $\Alg\funset{C}=\Alg\funset{D}$. 
In \cite{Pin17a, TotWal17} one finds examples of different clones that are algebraically equivalent; in fact Tóth and Waldhauser \cite{TotWal17} proved that on the two element set there are only finitely many algebraically inequivalent clones. Pinus showed that on each finite set there are only finitely many universal algebraic geometries closed under taking unions \cite{Pin17a, AicRos20}.
The aim of this note is to provide infinite families of algebraically inequivalent clones on a finite set.

Analysing the construction of $2^{\aleph_{0}}$ clones on $\{0,1,2\}$ in \cite{JanMuc59}, we obtain:
\begin{theorem}\label{teor:cont_many_alg_in_clones}
Let $A$ be a finite set with $\card{A}\geq 3$. Then there are $2^{\aleph_0}$ algebraically inequivalent clones on $A$.
\end{theorem}
The proof will be given in Section \ref{sec:3elementset}.

A clone $\funset{C}$ is called a \emph{Mal'cev clone} if it contains a ternary function such that for all $a,b\in A$,
\[
d(a,b,b)= d(b,b,a)=a,
\]
and it is called \emph{constantive} if every unary constant function on $A$ is in $\funset{C}$.
Using Idziak's construction from \cite{Idz99}, we prove:
\begin{theorem}\label{teor:infinitely_many_lgebraically_inquivalent_malcev_clones}
Let $A$ be a finite set with $\card{A}\geq 4$. Then there are exactly $\aleph_0$ algebraically inequivalent constantive Mal'cev clones on $A$. 
\end{theorem}
The  proof will be given in Section \ref{sec:Malcev}. It relies on the following fact:
\begin{theorem}\label{teor:inf_many_expansions_Zp2n}
Let $p$ be a prime, and let $n\in\N$. Then there are exactly $\aleph_0$ algebraically inequivalent clones that contain $\POL(\Z_{np^2},+)$.
\end{theorem}
The proof will be given in Section \ref{sec:znp2}.

In Table 1 we summarize our current knowledge on the number of algebraically inequivalent clones on a finite set.

\begin{table}
\begin{tabular}{C{65mm}C{30mm}C{35mm}}
\toprule
Property of the clone on an $n$-element set & Number of clones & Number of clones modulo $\algeq$\tabularnewline
\midrule
$n=2$ & $\aleph_0$ \cite{Pos41} & finite \cite{TotWal17}\tabularnewline 
$n \ge 3$ & $2^{\aleph_0}$ \cite{JanMuc59} & $2^{\aleph_0}$ (Theorem \ref{teor:cont_many_alg_in_clones})\tabularnewline
constantive, $n\geq 3$ & $2^{\aleph_0}$ \cite{AgoDemHan83} & ?\tabularnewline
equationally additive & ? & finite \cite{Pin17a,AicRos20} \tabularnewline
 
 Mal'cev, $n\leq3$ & finite \cite{Bul01} & finite \tabularnewline
 constantive, Mal'cev, $n\geq 4$  & $\aleph_0$ \cite{Aic10,Idz99} & $\aleph_0$ (Theorem  \ref{teor:infinitely_many_lgebraically_inquivalent_malcev_clones}) \tabularnewline
 contains $\Clo(\Z_n, +)$, $n$ squarefree & finite \cite{Fio21,AicMay07,May08}& finite\tabularnewline
  contains $\POL(\Z_{mp^2}, +)$, $m\geq 1$  & $\aleph_0$ \cite{Aic10, Idz99}& $\aleph_0$ (Theorem \ref{teor:inf_many_expansions_Zp2n})\tabularnewline
\bottomrule
\end{tabular}
\caption{The first column lists the properties of the clones we consider, the second column reports the number of such clones and the third column lists the number of such clones modulo $\algeq$. The number $p$ is assumed to be prime.
  Question marks represent open problems.}
\end{table}

\section{Notation}\label{sec:notation}
We write $\N$ for the set of positive integers and for $n\in \N$, $\finset{n}:=\{1,\dots, n\}$.
For a set $A$ the $i$-th component of $\vec{a}\in A^n$ is denoted by $a_i$ and $\vec{a}(i)$.
We will use the notions \emph{clone}, $\POL\ari{n}(\ab{A})$ and $\Clo\ari{n}(\ab{A})$ as they are commonly used in universal algebra \cite{BurSan81,MckMcnTay88}. For $k\in\N$ with $k\leq n$, we define the $k$-th $n$-ary projection $\pi_k^{(n)}:A^n\to A$ by $\pi_k^{(n)}(a_1,\dots,a_n)=a_k$ for all $\vec{a}\in A^n$. We set $\projections:=\{\pi_k^{(n)}\mid n\in\N \text{ and }k\leq n\}$. 
For a clone $\funset{C}$ on $A$, $\funset{C}\arii{n}$ is the set of $n$-ary functions in $\funset{C}$. 
For $k\in\N$, for $\sigma\colon \finset{k}\to\finset{n}$ and for $f\in\funset{C}\arii{k}$, we define $f_\sigma\in\funset{C}\arii{n}$ by $f_\sigma(x_1,\dots,x_n):=f(x_{\sigma(1)},\dots,x_{\sigma(k)})$ for all $x_1,\dots, x_n\in A$, and we call $f_\sigma$ a \emph{minor} of $f$. 
If $t$ is a $k$-ary term in the language of a universal algebra $\ab{A}$, we write $t_\sigma$ for the term $t(x_{\sigma(1)},\dots,x_{\sigma(k)})$. Observe that in this case $(t_\sigma)^\ab{A}= (t^\ab{A})_\sigma$. 

Let $A$ be a set and let $\funset{C}$ be a clone on $A$. For $n\in\N$ and for $X\subseteq A^n$, we define $\algclo{\funset{C}}{X}$ to be the intersection of all the elements of $\Alg\ari{n}\funset{C}$ that contain $X$ as a subset. Since the intersection of a collection of elements of $\Alg\ari{n}\funset{C}$ is an element of $\Alg\ari{n}\funset{C}$, we infer that $\algclo{\funset{C}}{X}\in\Alg\ari{n}\funset{C}$. 
The following lemma will be used to assess whether a set $X$ is algebraic with respect to a clone $\funset{C}$, which is equivalent to $X=\algclo{\funset{C}}{X}$. 
\begin{lemma}\label{lemma:basic_characterization_algebraic_closure}
Let $A$ be a set, let $\funset{C}$ be a clone on $A$, let $n\in \N$, let $X\subseteq A^n$, and let $\vec{a}\in A^n$. Then we have 
\[\vec{a}\in \algclo{\funset{C}}{X} \Longleftrightarrow \Bigl(\forall f,g\in \funset{C}\arii{n}: f\restrict{X}=g\restrict{X}\Rightarrow f(\vec{a})=g(\vec{a})\Bigr ).\]
\end{lemma}
\begin{proof}
  ``$\Leftarrow$'': 
  We assume that for all $f,g\in \funset{C}\arii{n}$ with $f\restrict{X}=g\restrict{X}$, we have $f(\vec{a})=g(\vec{a})$. We show that $\vec{a}\in \algclo{\funset{C}}{X}$. To this end, we show that for each $B\in \Alg\ari{n}\funset{C}$ with
  $X\subseteq B$, we have
  $\vec{a}\in B$. Since $B$ is algebraic with respect to $\funset{C}$, there exists a system of $\funset{C}$-equations $\{p_i\approx q_i\mid i\in I\}$ such that $B=\{\vec{x}\in A^n\mid \forall i\in I\colon p_i(\vec{x})=q_i(\vec{x})\}$. Since $X\subseteq B$,
  we have $p_i\restrict{X}=q_i\restrict{X}$
for all $i\in I$.
Therefore, the assumption yields $p_i(\vec{a})=q_i(\vec{a})$ for all $i\in I$,
and thus $\vec{a}\in B$.

``$\Rightarrow$'':
We  assume that there exist $f,g\in \funset{C}\arii{n}$ such that $f\restrict{X}=g\restrict{X}$ and $f(\vec{a})\neq g(\vec{a})$. Then $B:=\{\vec{x}\in A^n \mid f(\vec{x})=g(\vec{x})\}$ satisfies $B\in\Alg\ari{n}\funset{C}$, $X\subseteq B$, and $\vec{a}\notin B$. Thus, $\vec{a}\notin \algclo{\funset{C}}{X}$. 
\end{proof}
\section{Countably many algebraically inequivalent constantive expansions of $\Z_{np^2}$}\label{sec:znp2}
Let $\ab{A}=(A; +,-,0,(f_i)_{i\in I})$ be an expanded group. Following \cite{Aic14}, a function $f\colon A^{n}\to A$ is \emph{absorbing} if $f(a_1,\dots, a_n)=0$ for all $a_1, \dots, a_n \in A$ with $0\in\{a_1,\dots, a_n\}$.

\begin{lemma}\label{lemma:absorbing_polynomials}
Let $p$ be a prime, let $n,d,k\in \N$ with $d\geq 2$ and $k\geq d+1$, let $\ab{A}_d$ be the algebra $(\Z_{np^2};+,-,0,f_d)$, where $f_d\colon \Z_{np^2}^d\to \Z_{np^2}$ is defined by
\[f_d(x_1,\dots, x_d)=np\prod_{j=1}^{d} x_j \text{ for  }x_1,\dots, x_d \in \Z_{np^2},\]
and let $g\in \POL\ari{k}(\ab{A}_d)$. If $g$ is absorbing, then
$g$ is the constant zero function.
\end{lemma}
Lemma \ref{lemma:absorbing_polynomials} states that the algebra $\ab{A}_d$ is $d$-supernilpotent (cf. \cite{AicMud13}). We provide a proof that makes no use of the notion of supernilpotency. 
\begin{proof}
Let $\funset{C}_d$ be the set of all operations on $\Z_{np^2}$ that are induced by a sum of monomials of the form $ax_{i_1}^{n_1}\dots x_{i_r}^{n_r}$, where $n_1,\dots, n_r\in\N$, $x_{i_1},\dots, x_{i_r}$ are pairwise distinct variables, $a\in\Z_{np^2}$, and the following property is satisfied:
\begin{equation}
  \label{eq:referee_a_lemma_4} 
   (r=0) \text{ or } (r=n_1=1) \text{ or }  
     (2\leq n_1 +\dots + n_r\leq d \text{ and } np \text{ divides } a).
\end{equation}
It is clear that $\{+,-,0, f_d\}\subseteq \funset{C}_d$, that the projection clone satisfies $\projections[\Z_{np^2}] \subseteq \funset{C}_d$, and that $\funset{C}_d$ contains the constant functions. Furthermore, it is straightforward to verify that for $t_1,t_2,\dots, t_d\in\funset{C}_d$ we have that $t_1 +t_2\in\funset{C}_d$, $t_1 -t_2\in\funset{C}_d$ and $f_d(t_1,\dots, t_d)\in\funset{C}_d$.
Hence $\POL(\ab{A}_d)\subseteq \funset{C}_d$. We assume that $g\in\POL\ari{k}(\ab{A}_d)$ is absorbing, and consider the representation of $g$ as a sum of monomials, each one satisfying \eqref{eq:referee_a_lemma_4}. Since $g$ is absorbing, $g(0,x_2,\dots, x_k)$ is the constant $0$-function. Thus, the sum of those monomials of $g$ that do not involve $x_1$ induces the constant zero function. 
Therefore, there exists a representation of $g$ as a sum of monomials of the above form, each of them involving the variable $x_1$. 
A similar argument applies to all the other variables $x_2,\dots, x_k$. Therefore, there exists a (possibly empty) sum of monomials that induces the same polynomial function as $g$ such that each monomial involves all the variables $x_1,\dots, x_k$ and satisfies \eqref{eq:referee_a_lemma_4}.
If $x_{i_1}^{n_1} \dots x_{i_r}^{n_r}$ is such a monomial, then $r \ge k$,
and therefore $n_1 + \cdots +n_r \ge k$. Since $k \ge 3$,
\eqref{eq:referee_a_lemma_4} implies
$d \ge n_1 + \cdots +n_r \ge k$, contradicting the assumption $k \ge d+1$.
Hence $g$ is induced by the empty sum of monomials, and therefore it is the constant $0$-function. 
\end{proof}

\begin{proof}[Proof of Theorem \ref{teor:inf_many_expansions_Zp2n}]
For each $d\in \N \setminus \{1\}$, let $\ab{A}_d$ and $f_d$ be as in the statement of Lemma \ref{lemma:absorbing_polynomials}. For $d\in \N \setminus \{1\}$, for $m\in \N$ and for $X\subseteq \Z_{np^2}^m$, we use $\algclo{d}{X}$ as a shorthand for $\algclo{\POL(\ab{A}_d)}{X}$.
Let $l,i\in \N$ with $l>i\geq 2$. We claim that $\POL(\ab{A}_i)\not\sim_{\text{alg}}\POL(\ab{A}_l)$.
For proving this, we let
\[Q:=\{\vec{a}\in \Z_{np^2}^l\mid \exists s\in\finset{l}: \vec{a}(s)=0\}\]
and show that $\algclo{i}{Q}\neq \algclo{l}{Q}$; as a consequence $\POL(\ab{A}_i)\not\sim_{\text{alg}}\POL(\ab{A}_l)$.

First we prove that $(1,\dots,1)\notin\algclo{l}{Q}$. To this end, we consider the term equation $f_l(x_1,\dots, x_l)\approx 0$. We have that $f_l(1,\dots, 1)=np$ and $f_l\restrict{Q}=0$. Hence Lemma~\ref{lemma:basic_characterization_algebraic_closure} yields $(1,\dots,1) \not\in \algclo{l}{Q}$.

We now prove that $\algclo{i}{Q}=\Z_{np^2}^{l}$. 
To this end, let us consider an equation of the form $t_1(x_1,\dots, x_{l})\approx t_2(x_1,\dots,x_{l})$ with $t_1,t_2\in \POL\ari{l}(\ab{A}_i)$. This equation is equivalent to $(t_1-t_2) \, (x_1,\dots, x_l)\approx 0$. Hence it suffices to consider equations of the form $t(x_1,\dots,x_{l})\approx 0$.
If $t(x_1,\dots,x_{l})\approx 0$ is satisfied by all elements of $Q$, then  the function $t$ is absorbing  and has arity greater than $i$. Lemma~\ref{lemma:absorbing_polynomials} implies that $t$ is the constant zero function. Therefore, the equation $t(x_1,\dots, x_l)\approx 0$ is satisfied by all elements in $\Z_{np^2}^l$. Hence Lemma~\ref{lemma:basic_characterization_algebraic_closure} yields $(1,\dots 1)\in \algclo{i}{Q}$ and therefore $\algclo{i}{Q}\neq \algclo{l}{Q}$.

We conclude that $(\POL(\ab{A}_i))_{i\in \N\setminus\{1\}}$ is an infinite family of algebraically inequivalent clones.
\end{proof}

\section{Countably many algebraically inequivalent constantive Mal'cev clones on finite sets with at least 4 elements}\label{sec:Malcev}
We report a construction from {\cite[Section 3]{Idz99}}: Let $A$ be a finite set such that $\card{A}\geq 4$ and let $\funset{C}$ be a constantive clone on $A$ with a Mal'cev term $p$. 
The construction picks an element $1\in A$ and an element $0\notin A$ and constructs a new clone $\funset{C}^\oplus$ on the set $A^\oplus :=A \cup \{0\}$.

For all $k\in\N$ and for all $f\in \funset{C}$, the operation $f^\oplus\colon (A^\oplus)^k\to A^\oplus$ is defined by
\[
f^\oplus(x_1,\dots, x_k):=
\begin{cases}
f(x_1,\dots,x_k) & \text{ if } (x_1,\dots, x_k)\in A^k,\\
0 &\text{ otherwise.}
\end{cases}
\]

A binary operation $\cdot$ on $A^\oplus$ is defined by
\[
x\cdot y:=
\begin{cases}
0 &\text{ if } (x,y)\in A^2,\\
1 &\text{ if } x=y=0,\\
x &\text{ if } x\in A \text{ and } y=0,\\
y &\text{ if } x=0 \text{ and } y\in A.
\end{cases}
\]

Finally, $\funset{C}^\oplus$ is defined as the clone on $A^\oplus$ generated by $\{f^\oplus \mid f\in \funset{C}\}\cup \{\cdot\}$ and all unary constant operations.

\begin{lemma}\label{lemma:on_shape_of_C_oplus}
Let $A$ be a finite set with $\card{A}\geq 4$, let $\funset{C}$ be a constantive clone on $A$ and let $f\in \funset{C}^\oplus$. Then
  \begin{equation} \label{eq:fprop}
    f\restrict{A^k}=0 \text{ or there exists } \hat{f}\in \funset{C} \text{ such that }
    f\restrict{A^k}=\hat{f}.
  \end{equation}  
\end{lemma}
\begin{proof}
Let $K$ be the set of all unary constant functions on $A^\oplus$.
We will show that all projections on $A^\oplus$ satisfy~\eqref{eq:fprop}, and that for all $k,n \in \N$, for all $n$-ary \[ s \in \{f^\oplus \mid f\in \funset{C}\}\cup \{\cdot\} \cup K\] and for all $k$-ary $t_1, \ldots, t_n$ that satisfy~\eqref{eq:fprop}, also $s (t_1, \ldots, t_n)$ satisfies~\eqref{eq:fprop}. From this, we conclude that every function in $\funset{C}^\oplus$ satisfies~\eqref{eq:fprop}.
  
Clearly, all the projections satisfy \eqref{eq:fprop}. 

Let $k, n \in \N$, let $t_1, \ldots, t_n$ be $k$-ary functions on $A^{\oplus}$ that satisfy~\eqref{eq:fprop}, and let $s$ be an $n$-ary function on $A^\oplus$ from
$\{f^\oplus \mid f\in \funset{C}\}\cup \{\cdot\} \cup K$.
We will show that $s(t_1,\ldots, t_n)$ satisfies~\eqref{eq:fprop}.

We first consider the case that $s = f^\oplus$ with  $f\in\funset{C}\arii{n}$.
If there exists $i \in\finset{n}$ such that $t_i\restrict{A^k}=0$, then by the definition
of $f^\oplus$, $f^\oplus(t_1,\dots,t_n)\restrict{A^k}=0$, and hence $s(t_1, \ldots, t_n)$
satisfies~\eqref{eq:fprop}.
If for all $i\in\finset{n}$, there exists $\hat{t_i} \in \funset{C}\arii{k}$ such that $t_i\restrict{A^k}=\hat{t_i}$, then by the definition of $f^\oplus$ we have
\[s (t_1, \ldots, t_n)\restrict{A^k} =
f^\oplus(t_1,\dots,t_n)\restrict{A^k}=f(\hat{g_1},\dots, \hat{g_n})\in \funset{C}.\] 
Therefore $s(t_1, \ldots, t_n)$ satisfies~\eqref{eq:fprop}.

Now we consider the case that $s$ is the function $\cdot$ (and $n = 2$).
Let  $t_1,t_2$ be $k$-ary functions on $A^\oplus$ satisfying~\eqref{eq:fprop}.
We show that $t_1 \cdot t_2$ satisfies~\eqref{eq:fprop}.
Since $t_1, t_2$ satisfy~\eqref{eq:fprop}, we have to consider the following cases.  

\textbf{Case 1}: $t_1\restrict{A^k}= t_2\restrict{A^k}=0$:  In this case, the definition of $\cdot$ yields
$t_1 \cdot t_2 \restrict{A^k}= 1$. Since $\funset{C}$ is constantive, $t_1 \cdot t_2\restrict{A^k} \in \funset{C}$, and thus $t_1 \cdot t_2$ satisfies~\eqref{eq:fprop}.

\textbf{Case 2}: \emph{$t_1\restrict{A^k}= 0$ and there exists $\hat{t_2}\in\funset{C}\arii{k}$ such that $t_2\restrict{A^k}=\hat{t_2}$}: In this case, the definition of $\cdot$ yields
$t_1 \cdot t_2 \restrict{A^k}= \hat{t_2}$. This implies that $t_1 \cdot t_2\restrict{A^k} \in \funset{C}$, and thus $t_1 \cdot t_2$ satisfies~\eqref{eq:fprop}.

\textbf{Case 3}: \emph{$t_2\restrict{A^k}=0$ and there exists $\hat{t_1}\in\funset{C}\arii{k}$ such that $t_1\restrict{A^k}=\hat{t_1}$}: In this case, the definition of $\cdot$ yields
$t_1 \cdot t_2 \restrict{A^k}= \hat{t_1}$. This implies that $t_1 \cdot t_2\restrict{A^k} \in \funset{C}$, and thus $t_1 \cdot t_2$ satisfies~\eqref{eq:fprop}.

\textbf{Case 4}: \emph{There exist $\hat{g_1},\hat{g_2}\in \funset{C}\arii{k}$ such that $g_1\restrict{A^k}=\hat{g_1}$ and $g_2\restrict{A^k}=\hat{g_2}$}: In this case, the definition of $\cdot$ yields
$t_1 \cdot t_2 \restrict{A^k}= 0$, thus $t_1 \cdot t_2$ satisfies~\eqref{eq:fprop}.

Finally, let us consider the case $s\in K$ (and $n=1$). Let $a\in A^\oplus$ be the unique function value of $s$, and let $t$ be a $k$-ary function on $A^\oplus$ that satisfies \eqref{eq:fprop}. Clearly, $s(t)(x_1,\dots, x_k)=a$ for all $x_1,\dots, x_k\in A^\oplus$. If $a\in A$, since $\funset{C}$ is constantive, we have that $s(t)$ satisfies \eqref{eq:fprop}. If $a=0$, then $s(t)(x_1,\dots, x_k)=0$ for all $(x_1,\dots, x_k)\in A^k$. Thus, $s(t)$ satisfies \eqref{eq:fprop}. 
\end{proof}

\begin{lemma}\label{lemma:extending_restrictions}
Let $A$ be a finite set with $\card{A}\geq 4$, let $\funset{C}$ be a constantive clone on $A$, let $f\in(\funset{C}^\oplus)\arii{k}$, and let $\vec{a},\vec{b}\in A^{k}$. Then $f(\vec{a})=0$ if and only if $f(\vec{b})=0$.
\end{lemma}
\begin{proof}
If $f(\vec{a})\neq 0$, then $f\restrict{A^{k}}\neq  0$. Thus by Lemma \ref{lemma:on_shape_of_C_oplus}, there exists $\hat{f}\in\funset{C}$ such that $f\restrict{A^{k}}= \hat{f}$, and therefore $f(\vec{b})=\hat{f}(\vec{b})\in A$, which implies $f(\vec{b})\neq 0$. 
\end{proof}

\begin{lemma}{\cite[Section 3]{Idz99}}\label{lemma:result_from_Idziak99}
Let $A$ be a finite set with $\card{A}\geq 4$ and let $\funset{C}$ be a constantive Mal'cev clone on $A$. Then $\funset{C}^\oplus$ is a constantive Mal'cev clone on a set of cardinality $\card{A} + 1$. 
\end{lemma}

\begin{proposition}\label{prop:algebraic_equivalence_of_C_oplus}
Let $A$ be a finite set with $\card{A}\geq 4$, and let $\funset{C}_1,\funset{C}_2$ be two constantive Mal'cev clones on $A$. If $\funset{C}_1$ and $\funset{C}_2$ are not algebraically equivalent, then $\funset{C}_1^\oplus$ and $\funset{C}_2^\oplus$ are not algebraically equivalent.
\end{proposition}
\begin{proof}
Without loss of generality, we assume that there exists $B\subseteq A^k$ that is algebraic with respect to $\funset{C}_1$ but not with respect to $\funset{C}_2$. Note that then $B\neq A^k$. Since $\funset{C}_2$ is constantive, $B$ is not empty. 
Then there exists $\vec{a}\in A^k\setminus B$ such that $\vec{a}\in\algclo{\funset{C}_2}{B}$.

Let $B_0:=B\cup ((A^\oplus)^k\setminus A^k)$. 
We claim that 
\begin{equation}\label{eq:B0inAlgC1oplus}
B_0\in\Alg\funset{C}_1^\oplus. 
\end{equation}
To prove \eqref{eq:B0inAlgC1oplus}, let $\vec{c}\notin B_0$.  By Lemma~\ref{lemma:basic_characterization_algebraic_closure} it suffices to show that there are two functions $p_1,p_2\in \funset{C}_1^\oplus$ such that $p_1\restrict{B_0}= p_2\restrict{B_0}$ and $p_1(\vec{c})\neq p_2(\vec{c})$. Since $\vec{c}\in A^k\setminus B$ and $B$ is algebraic with respect to $\funset{C}_1$, Lemma~\ref{lemma:basic_characterization_algebraic_closure} implies that there are $f_1,f_2\in\funset{C}_1\arii{k}$ such that $f_1\restrict{B}= f_2\restrict{B}$ and $f_1(\vec{c})\neq f_2(\vec{c})$. We set $p_1=f_1^\oplus$ and $p_2=f_2^\oplus$. Then $p_1\restrict{B_0}= p_2\restrict{B_0}$ and $p_1(\vec{c})\neq p_2(\vec{c})$, which concludes the proof of \eqref{eq:B0inAlgC1oplus}. 

We now claim that
\begin{equation}\label{eq:B0isnotAlgC2oplus}
\vec{a}\in\algclo{\funset{C}_2^\oplus}{B_0}. 
\end{equation}
Suppose that $\vec{a}\notin \algclo{\funset{C}_2^\oplus}{B_0}$. Then Lemma~\ref{lemma:basic_characterization_algebraic_closure} implies that there are $f_1,f_2\in(\funset{C}_2^\oplus)\arii{k}$ such that $f_1\restrict{B_0}= f_2\restrict{B_0}$ and $f_1(\vec{a})\neq f_2(\vec{a})$. Since $B\subseteq B_0$ we infer that $f_1\restrict{B}= f_2\restrict{B}$. From Lemma \ref{lemma:extending_restrictions} and $B\neq \emptyset$, we deduce that $f_1\restrict{B}= f_2\restrict{B}\neq 0$. Moreover, by Lemma \ref{lemma:on_shape_of_C_oplus}, there are $\hat{f_1},\hat{f_2}\in \funset{C}_2\arii{k}$ such that $f_1\restrict{A^k}= \hat{f_1}$ and $f_2\restrict{A^k}= \hat{f_2}$. Therefore, $\hat{f_1}\restrict{B}=\hat{f_2}\restrict{B}$ and $\hat{f_1}(\vec{a})\neq \hat{f_2}(\vec{a})$. Hence by Lemma~\ref{lemma:basic_characterization_algebraic_closure}, $\vec{a}\notin \algclo{\funset{C}_2}{B}$, contradicting the choice of $\vec{a}$.

From \eqref{eq:B0inAlgC1oplus} and \eqref{eq:B0isnotAlgC2oplus} we obtain $\Alg\funset{C}_1^\oplus\neq \Alg\funset{C}_2^\oplus$.
\end{proof}

\begin{proof}[Proof of Theorem \ref{teor:infinitely_many_lgebraically_inquivalent_malcev_clones}]
We proceed by induction on $n=\card{A}$. 
For the base case $n=4$, we use Theorem~\ref{teor:inf_many_expansions_Zp2n} with $n=1$ and $p=2$. For the induction step we let $n\geq 4$ and $\{\funset{C}_i\mid i\in\N\}$ be a collection of pairwise algebraically inequivalent constantive Mal'cev clones on a set of cardinality $n$. Then Proposition~\ref{prop:algebraic_equivalence_of_C_oplus} and Lemma~\ref{lemma:result_from_Idziak99} imply that $\{\funset{C}_i^\oplus\mid i\in\N\}$ is a collection of pairwise algebraically inequivalent constantive Mal'cev clones on a set of cardinality $n+1$.
\end{proof}

\section{Continuously many algebraically inequivalent clones on sets with at least 3 elements}\label{sec:3elementset}
Let $Z=\{0,1,2\}$. For $n\in\N$ and for $i \le n$, a function $f : Z^n \to Z$ \emph{depends on its $i$-th argument} if there exist $\vec{a}\in Z^{n}$ and $b\in Z$ such that $f(\vec{a})\neq f(a_1,\dots, a_{i-1}, b, a_{i+1}, \dots , a_{n})$. If $f$ depends on exactly one of its arguments, then $f$ is called \emph{essentially unary}.
The \textit{essential arity} of $f$ is defined by $\essenar{f}:=\card{\{i\in\finset{n}: f\text{ depends on its }i\text{-th argument}\}}$.
Let $\funset{F}$ be a set of finitary functions on $Z$. 
We say that \emph{$f$ belongs to $\funset{F}$ up to inessential arguments} if there exist $l\in\N$, $1\leq i_1<\dots< i_l\leq n$ and a function $g\in\funset{F}\arii{l}$ such that $f(x_1,\dots, x_n)=g(x_{i_1},\dots, x_{i_l})$ for all $\vec{x}\in Z^n$. 
 
Let $D_1=\emptyset$ and for each $m\in\N\setminus\{1\}$, let $D_m:=\{(1,2,\dots, 2),\dots, (2,\dots, 2,1)\}$. 
For each $n\in \N$, let $F_n$ be the set consisting of all the functions $f\colon Z^n\to Z$ such that 
\begin{align}\label{eq:referee_a}
&f(Z^n)\subseteq \{0,1\}, \text{ and}\\ 
\label{eq:referee_b}
&f^{-1}(\{1\})\subseteq D_n. 
\end{align}
We set $F:=\bigcup_{n\in\N}F_n$ and 
\[
F':=\bigcup_{n\in\N}\{f\colon Z^n\to Z\mid f \text{ belongs to }F\text{ up to inessential arguments}\}\cup \projections[Z].
\]
\begin{lemma}\label{lemma:the_shape_of_F'}
The sets $F$ and $F'$ satisfy the following properties:
\begin{enumerate}
\item For all $f\in F$ with $\essenar{f}\leq 1$, we have $f= 0$.\label{eq:referee_c_originale}
\item For all $n\in\N$ and for all $f\in F\ari{n}$ with $\essenar{f}>0$, we have that
  $\essenar{f}=n$ and $n\geq 2$.\label{eq:referee_e}
\item For all $f\in F'$ with $\essenar{f}=1$, we have $f\in \projections[Z]$.\label{eq:no_essential_arity_one}
\item For all $l,n\in\N$, for all $g\in F_l$ and for all $\rho\colon\finset{l}\to\finset{n}$, we have $g_\rho\in F'_n$.\label{eq:minors_of_F_are_in_F'}
\item For all $k,n\in\N$, for all $f\in F'_k$ and for all $\sigma\colon\finset{k}\to\finset{n}$, we have $f_\sigma\in F'_n$. \label{eq:F'closed_under_minors}
\end{enumerate}
\end{lemma}
\begin{proof}
\eqref{eq:referee_c_originale} 
Let $n\in\N$, let $f\in F_n$ with $\essenar{f}\leq 1$ and let $\vec{x}\in Z^n$. We prove that $f(\vec{x})=0$. 
If $\essenar{f}\leq 1$, then there exists $i\in \finset{n}$ such that for all $j\in\finset{n}\setminus\{i\}$ the function $f$ does not depend on its $j$-th argument. 
Let us define $f'\in Z^Z$ by $f'(y)=f(y,\dots, y)$ for all $y\in Z$. We prove that $f'=0$. To this end, let $a\in Z$. Since $(a,\dots, a)\notin D_n$, we have $f'(a)=f(a,\dots, a)=0$. Moreover, since for all $j\neq i$, $f$ does not depend on its $j$-th argument, $f(\vec{x})=f(x_i,\dots, x_i)=f'(x_i)=0$. This concludes the proof of \eqref{eq:referee_c_originale}. 

\eqref{eq:referee_e}
Let $n\in\N$ and let $f\in F\ari{n}$. If $f$ is not constantly zero, then there exists $\vec{a}\in D_n$ such that $f(\vec{a})=1$. Changing any one of the coordinates of $\vec{a}$ to $0$, the value of $f$ changes from $1$ to $0$. Hence $f$ depends on all of its arguments and \eqref{eq:referee_c_originale} implies that it cannot be essentially unary. This concludes the proof of \eqref{eq:referee_e}.

\eqref{eq:no_essential_arity_one}
Let $f\in F'_n$ with $\essenar{f}=1$. Seeking a contradiction, we suppose that $f$ is not a projection.  Then $f$ belongs to $F$ up to inessential arguments. Thus, there exist $1\leq j_1<\dots<j_k\leq n$ and $\hat{f}\in F$ such that $f(x_1,\dots, x_n)=\hat{f}(x_{j_1},\dots, x_{j_k})$ for all $\vec{x}\in Z^n$. If $\essenar{\hat{f}}\leq 1$, then \eqref{eq:referee_c_originale} yields $\hat{f}=0$, and so $f$ is the constant $0$-function, contradicting $\essenar{f}=1$. If $\essenar{\hat{f}}>1$, then $\essenar{f}>1$, contradicting $\essenar{f}=1$. This concludes the proof of \eqref{eq:no_essential_arity_one}.

\eqref{eq:minors_of_F_are_in_F'}
Let us assume that the image of $\rho$ is $\{j_1,\dots, j_k\}\subseteq \finset{n}$, with $j_1<\dots <j_k$, and let us define
\[
h:=\Set{\bigl( (x_{j_1},\dots, x_{j_k}),g(x_{\rho(1)},\dots, x_{\rho(l)})\bigr)\mid (x_1,\dots, x_n)\in Z^n}.
\]
We first prove that $h$ is a functional relation. To this end, let $\vec{x},\vec{y}\in Z^n$ be such that $(x_{j_1},\dots, x_{j_k})=(y_{j_1},\dots, y_{j_k})$. We show that $g(x_{\rho(1)},\dots, x_{\rho(l)})=g(y_{\rho(1)},\dots, y_{\rho(l)})$. Clearly, if $(x_{j_1},\dots, x_{j_k})=(y_{j_1},\dots, y_{j_k})$, then $(x_{\rho(1)}, \ldots, x_{\rho(l)}) = (y_{\rho(1)}, \ldots, y_{\rho(l)})$,
and therefore $g(x_{\rho(1)},\dots, x_{\rho(l)}) = g(y_{\rho(1)},\dots, y_{\rho(l)})$.
This concludes the proof that $h$ is a functional relation. 
Now we prove that $h\in F_k$. Condition \eqref{eq:referee_a} is clearly satisfied since $g\in F_l$. We prove that $h$ satisfies \eqref{eq:referee_b}.
To this end, let $\vec{x}\in A^n$ with $g(x_{\rho(1)},\dots, x_{\rho(l)})=1$. We show that $(x_{j_1},\dots, x_{j_k})\in D_k$. Since $g\in F_l$ and $g(x_{\rho(1)},\dots, x_{\rho(l)})=1$, there exists $a\in \finset{l}$ such that $x_{\rho(a)}=1$ and
$x_{\rho(b)}=2$ for all $b\in\finset{l}\setminus \{a\}$.
Since $j_1<\dots<j_k$ we have $\card{\{r\in\finset{k}\colon j_r=\rho(a)\}}\leq 1$. Moreover, since $\{j_1,\dots, j_k\}$ is the image of $\rho$, we have $\card{\{r\in\finset{k}\colon j_r=\rho(a)\}}\geq 1$. Thus, $\card{\{r\in\finset{k}\colon j_r=\rho(a)\}}=1$. Hence there exists $i\in\finset{k}$ such that $x_{j_i}=x_{\rho(a)}=1$. Then for each $r\in\finset{k}\setminus\{i\}$ there exists $b\in \finset{l}\setminus\{a\}$ such that $j_r=\rho(b)$, and then $x_{j_r}=x_{\rho(b)}=2$. Thus, $(x_{j_1},\dots, x_{j_k})\in D_k$, and so $h$ satisfies \eqref{eq:referee_b}. This proves that $h\in F_k$. Since $j_1<\dots< j_k$ and $g_\rho(\vec{a})=h(a_{j_1}, \dots, a_{j_k})$ for all $\vec{a}\in A^n$, $g_\rho$ belongs to $F$ up to inessential arguments. Hence $g_\rho \in F_n'$. 
This concludes the proof of \eqref{eq:minors_of_F_are_in_F'}.

\eqref{eq:F'closed_under_minors}
Let $k,n\in\N$, let $f\in F'_k$ and let $\sigma\colon\finset{k}\to\finset{n}$. We prove that $f_\sigma\in F'_n$. 
By the definition of $F'$, there exist $l\in\N$, $g\in F_l$ and $\tau\colon\finset{l}\to\finset{k}$ injective and increasing such that $f(y_1,\dots, y_k)=g(y_{\tau(1)},\dots, y_{\tau(l)})$ for all $\vec{y}\in Z^k$. Therefore, for all $\vec{x}\in Z^n$ we have
\[
f_\sigma(x_1,\dots, x_n)=f(x_{\sigma(1)},\dots, x_{\sigma(k)})=g(x_{\sigma(\tau(1))},\dots, x_{\sigma(\tau(l))}).
\]
Since $g\in F_l$, \eqref{eq:minors_of_F_are_in_F'} implies that $g_{\sigma\circ\tau}\in F'_n$, and thus $f_\sigma\in F'_n$.
This concludes the proof of \eqref{eq:F'closed_under_minors}.
\end{proof}

\begin{lemma}\label{lemma:F'_is_a_clone}
The set $F'$ is a clone. 
\end{lemma}
\begin{proof}
We prove that $F'$ is closed under $\zeta$, $\tau$, $\Delta$, $\nabla$ and $\star$ as they are defined in \cite{Mal76} (cf. \cite{CouLeh12,PosKal79}). 
To this end, let $f\in F'_n$. Lemma~\ref{lemma:the_shape_of_F'}\eqref{eq:F'closed_under_minors} yields that $\mathop{\zeta} f\in F'_n$, $\mathop{\tau} f\in F'_n$, $\mathop{\Delta}f\in F'_{n-1}$, and $\mathop{\nabla} f\in F'_{n+1}$.
Let $g\in F'_m$ and let $h$ be the function defined by $h(\vec{x})=(f\star g)(\vec{x})=f(g(x_1,\dots, x_{m}),x_{m+1},\dots , x_{m+n-1})$ for all $\vec{x}\in Z^{m+n-1}$. We prove that $h$ belongs to $F'$.
If $f$ is constant, then $h$ is a  constant mapping with the same
function value as $f$. Hence $h$ is a minor of $f$, and thus by
Lemma~\ref{lemma:the_shape_of_F'}\eqref{eq:F'closed_under_minors},
$h \in F'$. If $f$ is a projection, then either $h$ is a projection or $h=g$. In both cases $h$ belongs to $F'$.
If $g$ is a projection, then $h$ is a minor of $f$. Hence Lemma~\ref{lemma:the_shape_of_F'}\eqref{eq:F'closed_under_minors} yields $h\in F'$.
Finally, let us assume that both $f$ and $g$ belong to $F$ up to inessential arguments and $f$ is not constant. Then there exist $1\leq j_1<\dots <j_{n'}\leq n$, $1\leq i_1<\dots <i_{m'}\leq m$, $\hat{f}\in F_{n'}$ and $\hat{g}\in F_{m'}$ such that for all $\vec{x}\in Z^n$ and for all $\vec{y}\in Z^m$ we have $f(\vec{x})=\hat{f}(x_{j_1},\dots, x_{j_{n'}})$ and $g(\vec{y})=\hat{g}(y_{i_1},\dots, y_{i_{m'}})$. If $j_1\geq 2$, then for all $\vec{x}\in Z^{m+n-1}$ we have $h(\vec{x})=\hat{f}(x_{m+j_1 -1},\dots, x_{m+j_{n'}-1})$. Thus, $h$ is a minor of $\hat{f}$ and Lemma~\ref{lemma:the_shape_of_F'}\eqref{eq:minors_of_F_are_in_F'} yields $h\in F'$.
Let us now assume that $j_1=1$ and let $\hat{h}$ be defined by $\hat{h}(x_1,\dots, x_{m'+n'-1})=\hat{f}(\hat{g}(x_1,\dots, x_{m'}),x_{m'+1},\dots, x_{m'+n'-1})$ for all $\vec{x}\in Z^{m'+n'-1}$. We now prove that $\hat{h}\in F$. Since $\hat{f}$ belongs to $F$, $\hat{h}$ satisfies \eqref{eq:referee_a}. We now prove that $\hat{h}$ satisfies \eqref{eq:referee_b}. To this end, let $\vec{a}\in Z^{m'+n'-1}$ be such that $\hat{h}(\vec{a})=1$. Then, since $\hat{f}$ satisfies \eqref{eq:referee_b}, we infer that $(g(a_1,\dots, a_{m'}),a_{m'+1},\dots, a_{m'+n'-1})\in D_{n'}$. Moreover, since $g$ satisfies \eqref{eq:referee_a} and \eqref{eq:referee_b}, we have that $(a_1,\dots, a_{m'})\in D_{m'}$ and $(a_{m'+1},\dots, a_{m'+n'-1})=(2,\dots, 2)$. Thus, there exists $r\in \finset{m'}$ such that $a_r=1$ and for all $k\in\finset{m'+n'-1}\setminus\{r\}$ we have $a_k=2$. This proves that $\vec{a}\in D_{m'+n'-1}$. Thus, $\hat{h}$ satisfies \eqref{eq:referee_b}, and so $\hat{h}\in F$. Moreover, we have that for all $\vec{x}\in Z^{m+n-1}$
\[
h(x_1,\dots,x_{m+n-1})=\hat{f}(\hat{g}(x_{i_1},\dots, x_{i_{m'}}), x_{m+j_2 -1}, \dots, x_{m+j_{n'}-1}).
\]
Therefore $h$ belongs to $F$ up to inessential arguments, and so $h\in F'$. Therefore, $F'$ is closed under $\star$. 
\end{proof}
We make use of the following construction from \cite{JanMuc59} (cf. \cite[Chapter 3]{PosKal79}). For each $i\in \N\setminus\{1\}$, the function $g_i\colon Z^i \to Z$ is defined by 
\[
g_i(x_1,\dots, x_i)=
\begin{cases}
1 & \text{ if } (x_1,\dots, x_i)\in D_i,\\
0 & \text{ otherwise}.
\end{cases}
\]
For each $I\subseteq \N\setminus\{1\}$, $\ab{Z}_{I}$ is defined as $(Z;(g_i)_{i\in I})$.
\begin{lemma}\label{lemma:the_shape_of_gi_functions}
Let  $I\subseteq\N\setminus \{1\}$, let $k\in \N$ and let $g\in \Clo\ari{k}(\ab{Z}_I)$. Then we have:
\begin{enumerate}
\item \label{item:f_constant}  If $g$ is constant, then $g=0$.
\item \label{item:f_unary} If $g$ is essentially unary, then $g$ is a projection.
\item \label{item:f_esnbinary} If $g$ depends exactly on the arguments $i_1<\dots<i_l$ with $l\geq 2$, then the following two conditions are satisfied:
\begin{enumerate}
\item For all $\vec{b}\in Z^{k}$ we have $g(\vec{b})\neq 2$. 
\item For all $\vec{a}\in Z^{k}$ with $(a_{i_1},\dots, a_{i_l})\notin D_{l}$, we have $g(\vec{a})=0$. 
\end{enumerate}
\end{enumerate}
\end{lemma}
\begin{proof}
Equations \eqref{eq:referee_a} and \eqref{eq:referee_b} imply that for all $i\in\N\setminus\{1\}$, $g_i\in F$. Thus, Lemma \ref{lemma:F'_is_a_clone} yields that for all $I\subseteq \N\setminus\{1\}$, $\Clo(\ab{Z}_I)\subseteq F'$.
Since $g\in\Clo\ari{k}(\ab{Z}_I)\subseteq F'$, then either $g\in\projections[Z]$ or $g$ belongs to $F$ up to inessential arguments. 
If $g$ is a constant, then it cannot be a projection. Hence it belongs to $F$ up to inessential arguments. Thus, by   Lemma~\ref{lemma:the_shape_of_F'}\eqref{eq:referee_c_originale}, is the constant zero. This proves \eqref{item:f_constant}.
We now prove \eqref{item:f_unary}. Let us assume that $g$ is essentially unary. Then by Lemma~\ref{lemma:the_shape_of_F'}\eqref{eq:no_essential_arity_one}, $g$ is a projection. This proves \eqref{item:f_unary}.
We now prove \eqref{item:f_esnbinary}. Let us assume that $g$ depends on its arguments $i_1<\dots< i_l$, let $\vec{b}\in Z^{k}$ and let $\vec{a}\in Z^{k}$ be such that $(a_{i_1},\dots, a_{i_l})\notin D_l$. 
Since $l\geq 2$, $g$ is not a projection. Since $g\in F'$, $g$ belongs to $F$ up to inessential arguments. This, together with Lemma~\ref{lemma:the_shape_of_F'}\eqref{eq:referee_e}, implies that there exists $\hat{g}\in F\ari{l}$ such that $g(x_1,\dots, x_k)=\hat{g}(x_{i_1},\dots, x_{i_l})$ for all $\vec{x}\in Z^k$, and $\hat{g}$ depends on all of its arguments. Since $\hat{g}\in F$, \eqref{eq:referee_a} implies that $g(\vec{b})=\hat{g}(b_{i_1},\dots, b_{i_l})\in\{0,1\}$. Moreover, since $(a_{i_1},\dots, a_{i_l})\notin D_{l}$, \eqref{eq:referee_b} and \eqref{eq:referee_a} yield $g(\vec{a})=\hat{g}(a_{i_1},\dots, a_{i_l})=0$. This proves \eqref{item:f_esnbinary}.
\end{proof}

\begin{proposition}\label{prop:iinIiffgibulletisalgebraicinZI}
Let $I\subseteq \N\setminus\{1\}$, let $i\in \N\setminus\{1\}$ and let $g_i^\bullet =\{(x_1,\dots, x_i,x_{i+1})\in Z^{i+1}\mid x_{i+1}=g(x_1,\dots, x_i)\}$. 
Then $g_i^\bullet \in\Alg \ab{Z}_I$ if and only if $i\in I$.
\end{proposition}

\begin{proof}
If $i\in I$, then $g_i^\bullet \in\Alg \ab{Z}_I$ by definition.
We now assume that $g_i^\bullet\in\Alg \ab{Z}_I$ and prove that $i\in I$.

If $g_i^\bullet\in \Alg \ab{Z}_{I}$, since $(1,\dots, 1)\notin g_i^\bullet$, Lemma~\ref{lemma:basic_characterization_algebraic_closure} implies that there exist $f_1, f_2\in \Clo\ari{i+1}(\ab{Z}_I)$ such that $f_1\restrict{g_i^\bullet}= f_2\restrict{g_i^\bullet}$ and $f_1(1,\dots, 1)\neq f_2(1,\dots, 1)$. We now prove that one of the two is a projection and the other depends on at least two of its arguments. Seeking contradictions, let us suppose that this is not the case. We distinguish cases accordingly to the essential arity of $f_1$ and $f_2$.

\textbf{Case 1}: \emph{$f_1$ and $f_2$ are constant}: In this case $f_1(1,\dots, 1)=f_2(1,\dots, 1)$. 

\textbf{Case 2}: \emph{$f_1$ is constant and $f_2$ is essentially unary}: In this case, Lemma \ref{lemma:the_shape_of_gi_functions} yields that $f_1$ is the constant $0$ function and $f_2$ is a projection. Let $\vec{a}=(1,2,\dots, 2,1)$. Since $g(a_1,\dots, a_i)=1$, $\vec{a}\in  g_i^\bullet$. Lemma~\ref{lemma:the_shape_of_gi_functions} yields $f_1(\vec{a})=0$ and $f_2(\vec{a})\in\{1,2\}$. Hence $f_1(\vec{a})\neq f_2(\vec{a})$. Thus, $f_1\restrict{g_i^\bullet}\neq f_2\restrict{g_i^\bullet}$.

\textbf{Case 3}: \emph{$f_1$ is essentially unary and $f_2$ is constant}: This case is symmetric to Case 2. 

\textbf{Case 4}: \emph{$f_1$ is constant and $f_2$ depends on at least two of its arguments}: In this case Lemma \ref{lemma:the_shape_of_gi_functions} yields $f_1(1,\dots,1)=0=f_2(1,\dots,1)$. 

\textbf{Case 5}: \emph{$f_1$ is depends on at least two of its arguments and $f_2$ is constant}: This case is symmetric to Case 4. 

\textbf{Case 6}: \emph{$f_1$ and $f_2$ are both essentially unary}: In this case Lemma \ref{lemma:the_shape_of_gi_functions}\eqref{item:f_unary} implies that both $f_1$ and $f_2$ are projections. Thus $f_1(1,\dots, 1)=1=f_2(1,\dots,1)$. 

\textbf{Case 7}: \emph{$f_1$ and $f_2$ both depend on at least two of their arguments}: In this case Lemma \ref{lemma:the_shape_of_gi_functions}\eqref{item:f_esnbinary} yields $f_1(1,\dots,1)=0=f_2(1,\dots,1)$.  

Thus, we have proved that one among $f_1$ and $f_2$ is a projection, while the other depends on at least two of its arguments.
Without loss of generality, let us assume that $f_2$ is a projection and that $f_1$ depends on at least two of its arguments.
Let
\[
T:=\{\vec{a}\in Z^{i+1}\mid a_{i+1}=1 \text{ and } (a_1,\dots, a_i)\in D_i\}.
\]
Note that $T\subseteq g_i^\bullet$.
Let $j\in\{1,\dots, i, i+1\}$ be such that $f_2(x_1, \dots, x_{i+1})=x_j$ for
all $\vec{x} \in Z^{i+1}$. We claim $j=i+1$. Seeking a contradiction, let us suppose that $j\leq i$. Let $\vec{x}\in T$ be such that $\vec{x}(j)=2$. Then, since $f_1\restrict{T}=f_2\restrict{T}$, we have $f_1(\vec{x})=2$. On the other hand, Lemma \ref{lemma:the_shape_of_gi_functions}\eqref{item:f_esnbinary} implies that $f_1(\vec{x})\neq 2$. Thus we get the desired contradiction and deduce that $f_2$ is the $(i+1)$-th projection.
Next, we prove that $f_1$ does not depend on its $(i+1)$-th argument. Seeking a contradiction, let us suppose that $f_1$ depends on its $(i+1)$-th argument. Then, since $f_1$ depends on at least two of its arguments, there exists $j\in \finset{i}$ such that $f_1$ depends also on its $j$-th argument. Let $\vec{a}\in T$ be such that $\vec{a}(j)=1$. Since $f_1\restrict{T}=f_2\restrict{T}$, we have that $f_1(\vec{a})=f_2(\vec{a})=a_{i+1}=1$. On the other hand, Lemma~\ref{lemma:the_shape_of_gi_functions} implies that $f_1(\vec{a})=0$ because $a_j=a_{i+1}=1$. From this
contradiction, we conclude that $f_1$ does not depend on its $(i+1)$-th argument.

Let $f\colon Z^i \to Z$ be defined by $f(x_1,\dots x_i):=f_1(x_1,\dots, x_i,x_1)$. 
Since $f_1$ does not depend on its $(i+1)$-th argument, for each $\vec{a}\in g_i^\bullet$, we have
\begin{equation}\label{eq:equazione_che_prove_la_inclusione_dei_graphi_di_f_e_gi}
f(a_1,\dots, a_i)= f_1(a_1,\dots, a_{i},a_1)=
f_1(a_1,\dots,a_{i+1})=f_2(a_{1},\dots,a_{i+1})= a_{i+1}.
\end{equation}
We show that $g_i=f$. To this end, let $\vec{a}\in Z^i$. Then $(a_1,\dots, a_i, g_i(\vec{a}))\in g_i^\bullet$, and thus \eqref{eq:equazione_che_prove_la_inclusione_dei_graphi_di_f_e_gi} yields $f(\vec{a})=g_i(\vec{a})$. 
Since $f\in \Clo\ari{i}(\ab{Z}_I)$, we deduce that $g_i\in\Clo\ari{i}(\ab{Z}_I)$. We now apply {\cite[Theorem 3.1.4]{PosKal79}}  and deduce that $i\in I$.
\end{proof}
\begin{corollary}\label{cor:twotothealephzeroclonesonthreelementset}
On the set $Z=\{0,1,2\}$ there are exactly $2^{\aleph_0}$ algebraically inequivalent clones. 
\end{corollary}
\begin{proof}
Let $\setofclone{Z}$ be the set of all clones on $Z$ and 
let $\psi\colon\potenza{\N\setminus\{1\}}\to\setofclone{Z}/{\algeq}$
be defined by $I\mapsto \Clo(\ab{Z}_I)/{\algeq}$. Proposition~\ref{prop:iinIiffgibulletisalgebraicinZI} yields that for all $I\neq J\in \potenza{\N\setminus\{1\}}$ we have $\Clo(\ab{Z}_I)\algeq \Clo(\ab{Z}_J)$ if and only if $I=J$. Therefore, $\psi$ is injective. Hence there are at least $2^{\aleph_0}$ algebraically inequivalent clones on $Z$. Since by \cite[Theorem 3.1.4]{PosKal79} there are at most $2^{\aleph_0}$ clones on $Z$, we can conclude that there are exactly $2^{\aleph_0}$ algebraically inequivalent clones on $Z$.  
\end{proof}

\begin{proof}[Proof of Theorem \ref{teor:cont_many_alg_in_clones}]
We first show that for a set $A$ and an element $u\notin A$, $A\cup\{u\}$ has at least as many algebraically inequivalent clones as $A$. To this end, let $A$ be a finite set with $\card{A}\geq 3$, let $\setofclone{A}$ be the set of all clones on $A$, let $u\notin A$, let $A^\oplus=A\cup\{u\}$, let $\setofclone{A^\oplus}$ be the set of all clones on $A^\oplus$, and let $\Phi\colon \setofclone{A}\to \setofclone{A^\oplus}$ be defined by
\[
\Phi(\funset{C})=\underset{n\in \N}{\bigcup}\{f:(A^\oplus)^{n}\to A^\oplus \,\mid\, f\restrict{A^n}\in\funset{C}\}.\]
We first prove that for all $\funset{C}\in\setofclone{A}$, $\Phi(\funset{C})$ is a clone on $A^\oplus$.
Clearly, all the projections belong to $\Phi(\funset{C})$. 
Let $f\in\Phi(\funset{C})\arii{n}$, let $t_1,\dots, t_n\in\Phi(\funset{C})\arii{k}$. Then 
\[
f(t_1,\dots, t_n)\restrict{A^{k}}=f(t_1\restrict{A^{k}},\dots, t_n\restrict{A^{k}})=f\restrict{A^{n}}(t_1\restrict{A^{k}},\dots, t_n\restrict{A^{k}}).
\]
The definition of $\Phi$ yields $f\restrict{A^{n}}, t_1\restrict{A^{k}},\dots, t_n\restrict{A^{k}}\in \funset{C}$. Thus, $f\restrict{A^{n}}(t_1\restrict{A^{k}},\dots, t_n\restrict{A^{k}})$ belongs to $\funset{C}$ and $f(t_1,\dots, t_n)\in\Phi(\funset{C})$.
Therefore, $\Phi(\funset{C})$ is a clone on $A^\oplus$.

Next, we prove that $\Phi$ satisfies 
\begin{equation}\label{eq:property_algebraic_sets_of_phi_C}
\forall \funset{C}\in \setofclone{A}, \forall n\in \N, \forall B\subseteq (A^\oplus)^n\colon B\in \Alg\Phi(\funset{C})\Leftrightarrow B\cap A^{n}\in \Alg \funset{C}.
\end{equation}
For proving \eqref{eq:property_algebraic_sets_of_phi_C}, let $\funset{C}\in \setofclone{A}$, let $n\in\N$, let $B\subseteq (A^\oplus)^n$ and let us assume that $B\in \Alg\ari{n}\Phi(\funset{C})$. We prove that $B\cap A^{n}\in\Alg\ari{n}\funset{C}$. 
To this end, let $\vec{b}\in A^{n}\setminus B$. By Lemma~\ref{lemma:basic_characterization_algebraic_closure} it suffices to prove that there are $\hat{f_1},\hat{f_2}\in \funset{C}$ such that $\hat{f_1}\restrict{A^{n}\cap B}=\hat{f_2}\restrict{A^{n}\cap B}$ and $\hat{f_1}(\vec{b})\neq \hat{f_2}(\vec{b})$. Since $B\in \Alg\ari{n}\Phi(\funset{C})$ and $\vec{b}\notin B$, Lemma~\ref{lemma:basic_characterization_algebraic_closure} yields that there exist $f_1,f_2\in \Phi(\funset{C})$ such that $f_1\restrict{B}=f_2\restrict{B}$ and $f_1(\vec{b})\neq f_2(\vec{b})$. Let $\hat{f_1}=f_1\restrict{A^{n}}$ and let $\hat{f_2}=f_2\restrict{A^{n}}$.  Then clearly $\hat{f_1},\hat{f_2}\in\funset{C}$, $\hat{f_1}\restrict{A^{n}\cap B}=\hat{f_2}\restrict{A^{n}\cap B}$ and $\hat{f_1}(\vec{b})\neq \hat{f_2}(\vec{b})$. Thus, $B\cap A^{n}\in \Alg \funset{C}$.

Let us now assume that $B\cap A^{n}\in \Alg \funset{C}$. We prove that $B\in \Alg\ari{n}\Phi(\funset{C})$. Let $\vec{a}\in (A^\oplus)^n\setminus B$. By Lemma~\ref{lemma:basic_characterization_algebraic_closure} it suffices to show that there are $f_1,f_2\in\Phi(\funset{C})$ such that $f_1\restrict{B}=f_2\restrict{B}$ and $f_1(\vec{a})\neq f_2(\vec{a})$. 
We split the proof into two cases.

\textbf{Case 1}: $\vec{a}\in A^{n}$: In this case Lemma~\ref{lemma:basic_characterization_algebraic_closure} implies that there are $g_1,g_2\in \funset{C}$ such that $g_1\restrict{B\cap A^n}=g_2\restrict{B\cap A^n}$ and $g_1(\vec{a})\neq g_2(\vec{a})$. Then for $i\in\{1,2\}$, we set
\[
f_i(\vec{x})=\begin{cases}g_i(\vec{x})&\text{ if } \vec{x}\in A^n,\\u &\text{ otherwise}.
\end{cases}
\]
By the definition of $\Phi(\funset{C})$, both $f_1$ and $f_2$ belong to $\Phi(\funset{C})$, $f_1\restrict{B}=f_2\restrict{B}$ and, since $\vec{a}\in A^{n}$, $f_1(\vec{a})\neq f_2(\vec{a})$.

\textbf{Case 2}: $\vec{a}\notin A^{n}$: Let $v_1,v_2\in A$ with $v_1\neq v_2$. For $i\in \{1,2\}$, we define 
\[
f_i(\vec{x})=
\begin{cases}
u &\text{ if } \vec{x}\in (A^\oplus)^n\setminus (A^{n}\cup \{\vec{a}\}),\\
x_1 &\text{ if } \vec{x}\in A^n,\\
v_i &\text{ if } \vec{x}\in \{\vec{a}\}.
\end{cases}
\]
By construction, we have that $f_1\restrict{A^n},f_2\restrict{A^{n}}\in \funset{C}$. Hence $f_1,f_2\in\Phi(\funset{C})$. Moreover, $f_1\restrict{B}=f_2\restrict{B}$ and $f_1(\vec{a})\neq f_2(\vec{a})$. 
Thus, we can conclude that $B\in\Alg\Phi(\funset{C})$.
Equation \eqref{eq:property_algebraic_sets_of_phi_C} implies that for all $\funset{C},\funset{D}\in\setofclone{A}$, if $\funset{D}\not\sim_{\text{alg}} \funset{C}$ then $\Phi(\funset{C})\not\sim_{\text{alg}} \Phi(\funset{D})$. 

Since by Corollary~\ref{cor:twotothealephzeroclonesonthreelementset}  there are $2^{\aleph_0}$ algebraically inequivalent clones on the set $\{0,1,2\}$ and since  there are at most $2^{\aleph_0}$ clones on a finite set, we deduce that on each finite set $A$ with $\card{A}\geq 3$ there are exactly $2^{\aleph_0}$ algebraically inequivalent clones. 
\end{proof}

\section*{Acknowledgements}
The authors thank
Tamás Waldhauser for discussions on the questions addressed in the present note and the referee, who has suggested the content of Lemmas \ref{lemma:absorbing_polynomials} and \ref{lemma:the_shape_of_gi_functions} to replace our originally more involved argument for proving Theorems \ref{teor:inf_many_expansions_Zp2n} and \ref{teor:cont_many_alg_in_clones}. 
\bibliographystyle{amsplain}
\bibliography{bibliografia.bib}
\end{document}